\documentclass[11pt]{amsart}

\newcommand{\private}[1]{}
\newcommand{\bfn}[1]{}                          
\newcommand{\ifn}[1]{}      
        
\usepackage{amssymb, amsmath, amscd, amsthm, color, epsfig,url, graphicx}

\usepackage[all]{xy}          % for xy-pic pictures
\xyoption{dvips}              % for xy-pic pictures

\makeatletter
\renewcommand\l@subsection{\@tocline{2}{0pt}{2pc}{5pc}{}}
\makeatother

\newcommand{\R}{{\mathbb R}}

\newcommand{\hofiber}{\operatorname{hofiber}}

\newcommand{\holim}{\operatorname{holim}}
\newcommand{\colim}{\operatorname{colim}}
\newcommand{\hocolim}{\operatorname{hocolim}}

\newcommand{\calF}{{\mathcal{F}}}
\newcommand{\calX}{{\mathcal{X}}}
\newcommand{\calY}{{\mathcal{Y}}}

\newcommand{\del}{{\partial}}

\theoremstyle{plain}
\newtheorem{thm}{Theorem}[section]

\newtheorem{lemma}[thm]{Lemma}

\theoremstyle{definition}
\newtheorem{defin}[thm]{Definition}

\newtheorem{def/ex}[thm]{Definition/Example}

\theoremstyle{remark}

\newcommand{\refT}[1]{Theorem~\ref{T:#1}}

\newcommand{\refL}[1]{Lemma~\ref{L:#1}}

\begin{document}

%%%%%%%%%%%%%%%%%%%%%%%%%%%%%%%%%%%%%%%%%%%%%%

\title[Blakers-Massey Theorem for $n$-cubes]{A purely homotopy-theoretic proof of the Blakers-Massey Theorem for $n$-cubes}

%%%%%%%%%%%%%%%%%%%%%%%%%%%%%%%%%%%%%%%%%%%%%%

\author{Brian A. Munson}
\address{Department of Mathematics, U.S. Naval Academy, Annapolis, MD}
\email{munson@usna.edu}
%\urladdr{???}

%\subjclass[2010]{Primary: 57Q45; Secondary: 57M27, 81Q30, 57R40}
%\keywords{configuration space integrals, links, homotopy links, finite type invariants, chord diagrams, weight systems}
%
%\thanks{Thanks.}

%%%%%%%%%%%%%%%%%%%%%%%%%%%%%%%%%%%%%%%%%%%%%%%%%%%%%%%%%%%%%%%%%%%%%%%%%%%%%%%%%%%%%%%%%%%%%%%%%%%%%%%%

\begin{abstract}
%{\bf Version: \today}
Goodwillie's proof of the Blakers-Massey Theorem for $n$-cubes relies on a lemma whose proof invokes transversality. The rest of his proof follows from general facts about cubes of spaces and connectivities of maps. We present a purely homotopy-theoretic proof of this lemma. The methods are elementary, using a generalization and modification of an argument originally due to Puppe used to prove the Blakers-Massey Theorem for squares.
\end{abstract}

\maketitle

%\tableofcontents

\parskip=4pt
\parindent=0cm

%%%%%%%%%%%%%%%%%%%%%%%%%%%%%%%%%%%%%%%%%%%%%%%%%%%%%%%%%%%%%%%%%%%%%%%%%%%%%%%%%%%%%%%%%%%%%%%%%%%%%%%%

\section{Introduction}\label{S:Intro}

%%%%%%%%%%%%%%%%%%%%%%%%%%%%%%%%%%%%%%%%%%%%%%%%%%%%%%%%%%%%%%%%%%%%%%%%%%%%%%%%%%%%%%%%%%%%%%%%%%%%%%%%

%check your version of the classic theorem. is it stated correctly?

Homology has proven a useful tool because it is often computable and produces interesting invariants. In contrast, homotopy groups are usually very difficult to compute. From the standpoint of the Eilenberg-Steenrod axioms for a homology theory, the difference in the computational difficulty can be explained by the fact that homology satisfies excision while homotopy does not. However, homotopy groups satisfy excision through a range of dimensions. The most fundamental result in this direction is the Freudenthal Suspension Theorem, which gives a range of dimensions in which the homotopy groups of a highly connected space $X$ are the same as the stable homotopy groups of $X$, the latter of which satisfy excision.

Recall that a space $X$ is called $j$-connected if every map $\del D^{i+1}=S^i\to X$ extends to a map $D^{i+1}\to X$ for $-1\leq i\leq j$. A map $f:Y\to X$ is $j$-connected if for all $x\in X$, the homotopy fiber $\hofiber_x(f)=\{(y,\gamma) : \gamma:I \to X, \gamma(1)=f(y),\gamma(0)=x\}$ is $(j-1)$-connected. In terms of pairs, $f:Y\to X$ is $j$-connected if $\pi_i(X,Y)=0$ for $1\leq i\leq j$ and the induced map $\pi_0Y\to\pi_0X$ is surjective (here $X$ should be replaced with the mapping cylinder of $f$ so that $Y$ is a subspace).

Suppose $X$ is a $j$-connected based space. The suspension of the identity map of $X$ is adjoint to a map $X\to\Omega\Sigma X$, and the Freudenthal Suspension Theorem says that this map is $(2j+1)$-connected. In terms of homotopy groups, the induced map $\pi_i (X)\to\pi_{i+1}\Sigma X$ is an isomorphism for $i\leq 2j$ and onto for $i=2j+1$. The stable homotopy groups $\pi_i^S X$ are defined as $\colim_{n} \pi_{i+n} \Sigma^nX$, so that $\pi_i X\to\pi^S_i X$ is an isomorphism in the range indicated. Hence the low-dimensional homotopy groups of $X$ can be replaced by groups satisfying the excision axiom.

The Freudenthal Suspension Theorem is itself a special case of the Triad Connectivity Theorem, also known as the Blakers-Massey Theorem for squares, which says that if a space $Y$ is the union of connected subspaces $Y_1$ and $Y_2$ along their connected intersection $Y_\emptyset=Y_1\cap Y_2$, and if $\pi_i(Y_1,Y_\emptyset)=0$ for $1\leq i\leq j$ and $\pi_i(Y_2,Y_\emptyset)=0$ for $1\leq i\leq l$, then the excision map $\pi_i(Y_1,Y_\emptyset)\to \pi_i(Y, Y_2)$ is an isomorphism for $1\leq i\leq j+l-1$ and onto for $i=j+l$ (to obtain the Freudenthal Suspension Theorem let $Y_\emptyset=X$ be $j$-connected and $Y_1,Y_2$ be copies of the cone on $Y_\emptyset$). We say ``for squares'' because the theorem can be neatly and more symmetrically described by organizing the spaces into the square diagram
$$\xymatrix{
Y_\emptyset \ar[r] \ar[d] & Y_1\ar[d]\\
Y_2 \ar[r] & Y,
}
$$
and the result can be interpreted as a range of dimensions in which either of the maps of pairs $(Y_1,Y_\emptyset)\to (Y,Y_2)$ or $(Y_2,Y_\emptyset)\to(Y, Y_1)$ induces isomorphisms in homotopy (i.e., a range in which these groups satisfy excision). Another more symmetric way to say this is that the map $Y_\emptyset\to\holim(Y_1\to Y\leftarrow Y_2)$ is $(j+l-1)$-conneced. This has generalizations to higher-order excision; for example, where $Y$ is the union of many spaces $Y_i$ along a common subspace $Y_\emptyset$.

The Blakers-Massey Theorem for $k$-cubes, also known as the $(k+1)$-ad Connectivity Theorem, is a result giving a range of dimensions in which higher-order excision holds for homotopy groups. A \emph{$k$-cube} of topological spaces is a functor $\calX=T\mapsto X_T$ from the poset of subsets of $\{1,\ldots, k\}$ to the category of topological spaces. Thus a $2$-cube $\calX$ is a square
$$\calX=
\xymatrix{
X_\emptyset\ar[r]\ar[d] & X_1\ar[d]\\
X_2\ar[r] & X_{12}.
}
$$
We say such a square is \emph{$j$-cocartesian} if the canonical map $\hocolim(X_2\leftarrow X_\emptyset\to X_1)\to X_{12}$ is $j$-connected. When $j=\infty$, we say the square is \emph{homotopy cocartesian}. A $k$-cube $\calX=T\mapsto X_T$ is called \emph{strongly cocartesian} if all its square faces are homotopy cocartesian. We say $\calX$ is \emph{$j$-cartesian} if the canonical map $X_\emptyset\to\holim_{\emptyset\neq T\subset\{1,\ldots, k\}}X_T$ is $j$-connected. See Section 1 and Definitions 1.3, 1.4, and 2.1 of \cite{CalcII} for terminology, and Section 2 for more on higher-order excision, including the results discussed in the current work.

\begin{thm}[Ellis-Steiner \cite{ES:BM}]\label{T:B-M}
If $\calX=T\mapsto X_T$ is a strongly cocartesian $k$-cube and the maps $X_\emptyset\to X_{\{i\}}$ are $j_i$-connected for all $i\in\{1,\ldots, k\}$, then $\calX$ is $(1-k+\sum_{i=1}^k j_i)$-cartesian.
\end{thm}

The proof of the Blakers-Massey Theorem for $k$-cubes is originally due to Barratt and Whitehead \cite{BW:n+1ad} with the additional hypothesis that $j_i\geq 2$ for all $i$, and was later improved as above by Ellis and Steiner \cite{ES:BM}. In addition, Ellis and Steiner were able to compute the first non-trivial group of such cubes. Their techniques use $\mathrm{cat}^k$-groups, following Brown and Loday \cite{BL:Excision, BL:VanKampen, BL:Hurewicz}, and as a result the proofs require extra machinery and are algebraic in nature. At the expense of losing information about the first non-trivial group, one can use the simpler and more direct space-level proof due to Goodwillie \cite{CalcII}, who was also able to prove a generalization of \refT{B-M} to a wider class of cubes (Theorem 2.5 of \cite{CalcII}). As Goodwillie notes, a good deal of his proof of \refT{B-M} is quite formal, relying on general results about cubical diagrams and connectivities of maps. However, it relies on a lemma, stated below, which uses a transversality (dimension counting) argument, and as such depends on arguments from the smooth category. Our note aims to prove the lemma using only elementary homotopy theory, and alongside Goodwillie's formal arguments stands as a purely homotopy-theoretic and space-level proof of \refT{B-M}. As with proofs of excision for homology, our techniques utilize subdivision (see for instance \refL{tech-lemma} below).

For a positive integer $k$, let $\underline{k}=\{1,\ldots, k\}$. Let $\calX=T\mapsto X_T$ be a pushout cube of spaces, formed by attaching cells $e_j$ of dimension $d_j+1$ for $1\leq j\leq k$ to a space $X_\emptyset$. That is, $X_T=X_\emptyset\cup\{e_j : j\in T\}$ for $T\subset\underline{k}$ for some choice of attaching maps $\del e_j\to \calX(\emptyset)$, $1\leq j\leq k$.

\begin{lemma}\label{L:Goodwillielemma}[Lemma 2.7 of \cite{CalcII}]
With $\calX$ as above, choose a basepoint $x\in X_{\{k\}}$, and for $T\subset\underline{k-1}$ let $\calF(T)=\hofiber_x(X_T\to X_{T\cup \{k\}})$. Then the $(k-1)$-cube $\calF$ is $(-1+\sum_{j=1}^k d_j)$-cocartesian. That is, the pair $\left(\calF(\underline{k-1}),\cup_{j\in \underline{k-1}}\calF(\underline{k-1}-j)\right)$ is $(-1+\sum_{j=1}^k d_j)$-connected.
\end{lemma}

We learned an elementary proof of the Triad Connectivity Theorem (\refT{B-M} in the case of squares) from tom Dieck's book \cite{tD:AT}, who credits Puppe \cite{tDKP:H}. The main theme of this proof is subdivision, much like proofs of excision for homology. We adapt these ideas to prove \refL{Goodwillielemma} without use of transversality arguments. Our proof mirrors Goodwillie's, and we replace his ``dimension counting'' argument with a ``coordinate counting'' one. 

% We have briefly sketched or entirely omitted arguments when adapting those appearing in tom Dieck \cite{tD:AT} are straightforward.

%For a positive integer $k$, write $\underline{k}=\{1,\ldots, k\}$. Suppose $\calY$ is a $k$-cube of spaces, with $\calY(T)=Y_T$ for $T\subset\underline{k}$ such that $Y_T\to Y_{\underline{k}}$ is the inclusion of an open set, and such that $Y_T\cap Y_U=Y_{T\cap U}$ and $Y_T\cup Y_U=Y_{T\cup U}$ for all $T,U\subset\underline{k}$. Thus $T\mapsto Y_T$ is a strongly cocartesian $k$-cube (see 2.1, \cite{CalcII}).
%
%Suppose $Y_0$ is a space, with cofibrations $Y_0\to Y_i$ for each $i=1$ to $k$, and let $Y_S$ be the union along $Y_0$ of the $Y_i$ for $i$ in $S\subset \{1,\ldots, k\}$ (thus the $k$-cube $S\mapsto Y_S$ is strongly cocartesian. See \cite{} for definitions).
%
%\begin{thm}\label{T:BMaltlemma}
%With $Y_T$ as above, choose a basepoint $y\in Y_\emptyset$, and for $T\subset\{1,\ldots, k-1\}$ let $\calF(T)=\hofiber(Y_T\to Y_{T\cup\{k\}})$. If $(Y_j,Y_0)$ is $d_j$-connected and $d_j\geq 0$ for all $j$, then the $(k-1)$-cube $T\mapsto \calF(T)$ is $(-1+\sum_{j=1}^kd_j)$-cocartesian.
%\end{thm}

%%%%%%%%%%%%%%%%%%%%%%%%%%%%%%%%%%%%%%%%%%%%%%%%%%%%%%%%%%%%%%%%%%%%%%%%%%%%%%%%%%%%%%%%%%%%%%%%%%%%%%%%

\section{Preliminaries}\label{S:prelims}

%%%%%%%%%%%%%%%%%%%%%%%%%%%%%%%%%%%%%%%%%%%%%%%%%%%%%%%%%%%%%%%%%%%%%%%%%%%%%%%%%%%%%%%%%%%%%%%%%%%%%%%%

%maybe shortest way to achieve the reduction to attaching cells of dimension greater than zero.
We first make a simplification in the hypotheses of \refL{Goodwillielemma}. If $d_j=-1$ for all $j$, then the conclusion of \refL{Goodwillielemma} is vacuously true. Without loss of generality $d_k\geq 0$. The basepoint $x\in X_{\{k\}}$ can be joined by a path to some point in $X_\emptyset$, so we may as well assume the basepoint lies in $X_\emptyset$ by the homotopy invariance of homotopy fibers over path components. If $d_i=-1$ for any other value of $i$, then $X_\emptyset\to X_{\{i\}}$ is the inclusion of $X_\emptyset$ to itself with a disjoint point added. This point plays no role in any of the homotopy fibers appearing in the cube $\calF$, and we may ignore it altogether. More precisely, for this value of $i$ and a basepoint $x\in X_\emptyset$, we have $\hofiber_x(X_T\to X_{T\cup\{k\}})=\hofiber_x(X_{T\setminus\{i\}}\to X_{T\setminus\{i\}\cup\{k\}})$ for all $T\subset\underline{k-1}$. Thus we may assume $d_j\geq0$ for all $1\leq j\leq k$. The remainder of this section is an adaptation of material from Section 6.9 of \cite{tD:AT}.

\begin{defin}
A \emph{cube} $W$ in $\R^n$ is a set of the form
$$
W=W(a,\delta,L)=\left\{x\in\R^n : a_i\leq x_i\leq a_i+\delta\mbox{ for }i\in L, x_i=a_i\mbox{ for }i\notin L\right\},
$$
where $a=(a_1,\ldots, a_n)\in\R^n$, $\delta>0$, and $L\subset\{1,\ldots, n\}$ (possibly empty). Define $\dim(W)=|L|$. The \emph{boundary} $\del W$ of $W$ is the set of all $x$ in $W$ such that $x_i=a_i$ or $x_i=a_i+\delta$ for at least one value of $i\in L$. The boundary $\del W$ is a union of \emph{faces}. A face of a cube is also a cube.
\end{defin}

\begin{defin}
With $W$ as above, for each $j=1$ to $k$ define
$$
K^{j,k}_{p}(W)=\left\{x\in W:\frac{\delta(j-1)}{k}+a_i<x_i<\frac{\delta j}{k}+a_i\mbox{ for at least $p$ values of $i\in L$}\right\}.
$$
\end{defin}

If $p\leq q$, then $K^{j,k}_q(W)\subset K^{j,k}_p(W)$. The following lemma gives the basic technical deformation result, with statement and proof a straightforward generalization of 6.9.1 of \cite{tD:AT}.

\begin{lemma}\label{L:tech-lemma}
Let $Y$ be a space with a subspace $A\subset Y$, $W$ a cube, $j,k$ positive integers with $j\leq k$, and $f:W\to Y$ a map. Suppose that for $p\leq\dim(W)$ we have
$$
f^{-1}(A)\cap W'\subset K^{j,k}_p(W')
$$
for all cubes $W'\subset \del W$. Then there exists a map $g:W\to Y$ homotopic to $f$ relative to $\del W$ such that 
$$
g^{-1}(A)\subset K^{j,k}_p(W).
$$
\end{lemma}

%short proof
%\begin{proof}
%Without loss of generality, $W=[0,1]^n$. Define $h$ to be the linear expansion of $[\frac{j-1}{k},\frac{j}{k}]^n$ into $[0,1]^n$ along rays centered at $(\frac{2j-1}{2k},\ldots, \frac{2j-1}{2k})$, and let $g=f\circ h$. The proof of 6.91 in \cite{tD:AT} can now be adapted to prove the result.
%\end{proof}

%long proof
\begin{proof}
Without loss of generality $W=I^n$, $n\geq 1$. We will construct a map $h:I^n\to I^n$ homotopic to the identify and define $g$ to be the composition of $f$ with $h$. Let $x=\left(\frac{2j-1}{2k},\ldots, \frac{2j-1}{2k}\right)$ be the center of the cube $\left[\frac{j-1}{k},\frac{j}{k}\right]^n$. For a ray $y$ emanating from $x$, let $P(y)$ be its intersection with $\del\left[\frac{j-1}{k},\frac{j}{k}\right]^n$ and $Q(y)$ its intersection with $\del I^n$. Let $h$ map the segment from $P(y)$ to $Q(y)$ onto the point $Q(y)$ and the segment from $x$ to $P(y)$ affinely onto the segment from $x$ to $Q(y)$. Clearly $h$ is homotopic to the identity of $I^n$ relative to $\del I^n$, and so $g=f\circ h$ is homotopic to $f$ relative to $\del I^n$. It remains to check that $g$ satisfies the property in the conclusion of the lemma.

Suppose $z\in I^n$ and $g(z)\in A$. Write $z=(z_1,\ldots, z_n)$. If $z\in \left(\frac{j-1}{k},\frac{j}{k}\right)^n$, then $z\in K^{j,k}_n(W)\subset K^{j,k}_p(W)$ and we are done. Now assume there exists $i$ so that either $z_i\geq \frac{j}{k}$ or $z_i\leq \frac{j-1}{k}$. Then by definition of $h$, we have $h(z)\in\del I^n$, so $h(z)\in W'$ for some face $W'$ of dimension $n-1$. Since $g(z)=f(h(z))\in A$, $h(z)\in f^{-1}(A)$, and by assumption then $h(z)\in K^{j,k}_p(W')$. Thus for at least $p$ values of $i$, we have $\frac{j-1}{k}<h(z)_i<\frac{j}{k}$, where $h(z)_i$ denotes the $i$th coordinate of $h(z)$. By definition of $h$,
$$
h(z)_i=\frac{2j-1}{2k}+t\left(z_i-\frac{2j-1}{2k}\right)\quad\mbox{ for }t\geq 1.
$$
Inserting this expression into the previous inequalities and solving for $z_i$ yields
$$
-\frac{1}{t2k}+\frac{2j-1}{2k}<z_i<\frac{1}{t2k}+\frac{2j-1}{2k}.
$$
Since the lower bound increases with $t$ and the upper bound decreases with $t$, substituting $t=1$ into each gives
$$
\frac{j-1}{k}<z_i<\frac{j}{k}
$$
so that $z\in K^{j,k}_p(W)$.
\end{proof}

Suppose $Y$ is a space with open subsets $Y_\emptyset,Y_1,\ldots, Y_k$ such that $Y$ is the union of $Y_1,\ldots, Y_k$ along $Y_\emptyset$. Let $f:I^n\to Y$ be a map. By the Lebesgue Covering Lemma, we can decompose $I^n$ into cubes $W$ such that $f(W)\subset Y_j$ for some $j$ depending on $W$. The following is a generalization of Theorem 6.9.2 in \cite{tD:AT}, as is its proof.

\begin{thm}\label{T:BMpretheorem}
With the $Y_j$ and $f$ as above, assume that for each $j$, $(Y_j,Y_\emptyset)$ is $d_j$-connected, with $d_j\geq 0$ (i.e., $Y_\emptyset\to Y_j$ is $d_j$-connected).\bfn{$(X,A)$ is $k$-connected is same thing as $A\to X$ is $k$-connected.} Then there is a homotopy $f_t$ of $f$ such that
\begin{enumerate}
\item If $f(W)\subset Y_j$, then $f_t(W)\subset Y_j$ for all $t$,
\item If $f(W)\subset Y_\emptyset$, then $f_t(W)=f(W)$ for all $t$, and
\item If $f(W)\subset Y_j$, then $f_1^{-1}(Y_j\setminus Y_\emptyset)\cap W\subset K^{j,k}_{d_j+1}(W)$.
\end{enumerate}
\end{thm}

%identical to the proof presented in tom Dieck.

%proof, for the sake of completeness.
\begin{proof}
Let $C^l$ be the union of cubes $W$ with $\dim(W)\leq l$. The homotopy $f_t$ is constructed inductively over $C^k\times I$. If $\dim(W)=0$, then if $f(W)\subset Y_\emptyset$, we simply let $f_t=f$, which achieves the second condition. Note that if $f(W)\subset Y_j\cap Y_i$ for $i\neq j$, then $f(W)\subset Y_\emptyset$. Hence we only need consider what happens if $f(W)\subset Y_j$ and $f(W)\not\subset Y_i$ for all $i\neq j$. In this case, if $f(W)\subset Y_j$ and $f(W)\not\subset Y_l$ for all $l\neq j$, then since $(Y_j,Y_\emptyset)$ is $d_j$-connected and $d_j\geq 0$, choose a path from $f(W)$ to some point in $Y_\emptyset$ and use this as the homotopy, so that $f_1(W)\subset Y_\emptyset$. Then clearly the first condition holds and so does the third (in this case, the third condition says the empty set is contained in $K^{j,k}_{d_j+1}(W)$). This proves the base case.

Since the inclusion $\del W\subset W$ is a cofibration for any cube $W$, we may extend this homotopy over all cubes $W$ so that the first and second conditions hold.\bfn{there is no obstruction; this is a standard extension problem. the only difference is one of the extensions is extending by a constant, but this is trivial} By induction suppose that $f$ has been changed by a homotopy satisfying all three conditions for cubes of dimension less than $l$, and let $W$ be a cube with $\dim(W)=l$. If $f(W)\subset Y_\emptyset$, we let $f_t=f$ as usual. If $f(W)\subset Y_j$ and $f(W)\not\subset Y_i$ for all $i\neq j$, then 
\begin{itemize}
\item if $\dim(W)=l\leq d_j$, then since $(Y_j,Y_\emptyset)$ is $d_j$-connected there is a homotopy $f_t$ of $f$ relative to $\del W$ such that $f_1(W)\subset Y_\emptyset$, and clearly the first and third conditions hold.
\item If $\dim(W)=l>d_j$, we employ \refL{tech-lemma}. Let $A=Y_j\setminus Y_\emptyset\subset Y_j$. By induction we have that, for all $W'\subset \del W$,
$$
f^{-1}(Y_j\setminus Y_\emptyset)\cap W'\subset K^{j,k}_l(W')\subset K^{j,k}_{d_j+1}(W'),
$$
and by \refL{tech-lemma}, there is a homotopy $f_t$ of $f$ relative to $\del W$ such that $f^{-1}_1(Y_j\setminus Y_\emptyset)\cap W\subset K^{j,k}_{d_j+1}(W)$.
\end{itemize}
\end{proof}

%%%%%%%%%%%%%%%%%%%%%%%%%%%%%%%%%%%%%%%%%%%%%%%%%%%%%%%%%%%%%%%%%%%%%%%%%%%%%%%%%%%%%%%%%%%%%%%%%%%%%%%%

\section{Proof of \refL{Goodwillielemma}}\label{S:proof}

%%%%%%%%%%%%%%%%%%%%%%%%%%%%%%%%%%%%%%%%%%%%%%%%%%%%%%%%%%%%%%%%%%%%%%%%%%%%%%%%%%%%%%%%%%%%%%%%%%%%%%%%

We need to convert the strongly cocartesian cube $\calX$ in the statement of \refL{Goodwillielemma} into one where the maps are inclusions of open sets in order to apply the previous results. For each $1\leq j\leq k$ and corresponding cell $e_j$ with attaching map $f_j:\del e_j\to X_\emptyset$, assume $e_j=D^{d_j+1}$, put $N_j=D^{d_j+1}-\{0\}$, and let $V_j$ be the interior of $D^{d_j+1}$. Define a $k$-cube $\calY=S\mapsto Y_S$ for $S\subset\underline{k}$ as follows. Let $U=X_\emptyset\cup_{j=1}^nN_j$. The inclusion $X_\emptyset\to U$ is a homotopy equivalence, and $U$ is open in $X_{\underline{k}}$. For $S\subset\underline{k}$, let $Y_S=U\cup_{j\in S} V_j$. Then $Y_S$ is open in $Y_{\underline{k}}=X_{\underline{k}}$ for each $S$, and the maps $Y_S\to Y_T$ for $S\subset T$ are the evident inclusions. The inclusion $X_S\to Y_S$ gives rise to a map of $k$-cubes $\calX\to \calY$ which is an equivalence for each $S$. Now we are ready to prove \refL{Goodwillielemma}.

%explaining the above choices for squares.
%Let $X$ be a space and $f:\del e^n\to X$ an attaching map for an $n$-cell. We obtain a homotopy cocartesian square
%$$\xymatrix{
%\del e^n \ar[r]\ar[d]_f & e^n\ar[d]\\
%X\ar[r] & X\cup_f e^n.
%}
%$$
%Assume $e^n=D^n$ is the closed unit disk in $\R^n$. Let $N=D^n-\{0\}$, $U=X\cup_fN$, and let $V$ be the interior of $D^n$. Then $U$ and $V$ are open in $X\cup_f D^n$, $U\cup V=X\cup_f D^n$, $U\simeq X$, $V\simeq D^n$, and $U\cap V\simeq \del D^n$. We have a strict map of squares from the diagram above to
%$$\xymatrix{
%U\cap V \ar[r]\ar[d] & V\ar[d]\\
%U\ar[r] & U\cup V.
%}
%$$
%which is an equivalence in each entry. More generally suppose we are in the situation described in the hypotheses of \refL{Goodwillielemma}. 

\begin{proof}[Proof of \refL{Goodwillielemma}]
With $\calY=T\mapsto Y_T$ as above, choose a basepoint $y\in Y_\emptyset$, put $\calF'(T)=\hofiber_y(Y_T\to Y_{T\cup\{k\}})$ for $T\subset\underline{k-1}$, and let $C$ be the contractible space $C=\hofiber_y(Y_{\underline{k-1}}\to Y_{\underline{k-1}})$. As indicated in the paragraph preceding this proof, $\calF'(T)\simeq\calF(T)$. Following Goodwillie's proof of \refL{Goodwillielemma} and using 1.16(a) in \cite{CalcII}, it is enough to show that the cube $T\mapsto\calF^\ast(T)=\calF'(T)\cup C$ is $(-1+\sum_j d_j)$-cocartesian; that is, that the pair
$$
\left(A,B\right)=\left(\calF^\ast(\underline{k-1}),\cup_{j\in\underline{k-1}}\calF^\ast(\underline{k-1}-j)\right)
$$
is $(-1+\sum_jd_j)$-connected. Note that the conclusion is automatic if $d_j=0$ for all $j$, since any pair $(A,B)$ is $(-1)$-connected. Let $\phi:(I^n,\del I^n)\to (A,B)$ be a map. The map $\phi$ is adjoint to a map $\Phi:I^n\times I\to Y_{\underline{k}}$ with boundary conditions
\begin{enumerate}
\item[(B0)] $\Phi(z,0)=y\in Y_\emptyset$ is the basepoint for all $z\in I^n$,
\item[(B1)] $\Phi(z,1)\in\cup_{j\in \underline{k-1}}Y_j=Y_{\underline{k-1}}$ for all $z\in I^n$, and
\item[(B2)] For each $z\in \del I^n$ there exists $i(z)\in \underline{k}$ so that $\Phi(z,t)\in \cup_{j\neq i(z)} Y_j$ for all $t\in I$.
\end{enumerate}
We will make a homotopy of $\Phi$ preserving (B0)-(B2) such that the last condition holds for each $z\in I^n$. To do this we apply \refT{BMpretheorem} to $\Phi:I^n\times I\to Y_{\underline{k}}$ and obtain a decomposition of $I^n\times I$ into cubes $W$ such that for each $W$ there is some $j$ so that $\Phi(W)\subset Y_j$, and a homotopy $\Phi_r$ for $0\leq r\leq 1$ of $\Phi=\Phi_0$ such that 
\begin{enumerate}
\item $\Phi(W)\subset Y_j$ implies $\Phi_r(W)\subset Y_j$ for all $r$,
\item $\Phi(W)\subset Y_\emptyset$ implies $\Phi_r(W)=\Phi(W)$ for all $r$, and
\item $\Phi(W)\subset Y_j$ implies $\Phi_1^{-1}(Y_j\setminus Y_\emptyset)\cap W\subset K^{j,k}_{d_j+1}(W)$.
\end{enumerate}
First we prove that $\Phi_r$ satisfies (B0)-(B2) for all $r$. 
\begin{itemize}
\item[(B0)] Since $\Phi(z,0)=y\in Y_\emptyset$ is the basepoint for all $z\in I^n$, we have for all cubes $W\subset I^n\times\{0\}$ that $\Phi(W)=y$, and (2) above implies $\Phi_r(W)=\Phi(W)$ for all $r$, so that $\Phi_r(z,0)=y$ for all $r$.
\item[(B1)] Since $\Phi(z,1)\in\cup_{j\in\underline{k-1}}Y_j=Y_{\underline{k-1}}$ for all $z\in I^n$, then for all cubes $W\subset I^n\times\{1\}$, $\Phi(W)\subset Y_j$ for some $1\leq j\leq k-1$. Hence $\Phi_r(W)\subset Y_j$ as well by (1) above, and thus $\Phi_r(z,1)\subset Y_{\underline{k-1}}$ for all $r,z$.
\item[(B2)] We know that for each $z\in \del I^n$ there exists $i(z)\in\underline{k}$ so that $\Phi(\{z\}\times I)\subset\cup_{j\neq i(z)} Y_j$. Let $W_1,\ldots, W_h$ be cubes such that $\{z\}\times I\subset W_1\cup\cdots \cup W_h$ and so that each $W_a$ contains a point of the form $(z,t)$ for some $t$. Since $\Phi(\{z\}\times I)\subset\cup_{j\neq i(z)} Y_j$, for each $a=1$ to $h$ we must have $\Phi(W_a)\subset Y_{j(a)}$ for some $j(a)\neq i(z)$.\bfn{which would otherwise contradict the fact that $\Phi(z,t)\notin Y_s$} This implies $\Phi_r(W_1\cup\cdots\cup W_h)\subset Y_{j(1)}\cup\cdots\cup Y_{j(h)}\subset\cup_{j\neq i(z)}Y_j$ for all $r$, again by (1) above.
\end{itemize}
Now we show that $\Phi_1$ actually satisfies the stronger condition that for each $z\in I^n$ there exists $i(z)\in \underline{k}$ so that $\Phi_1(z,t)\in \cup_{j\neq i(z)} Y_j$ for all $t\in I$ when $n<\sum_jd_j$. Let $\pi:I^n\times I\to I^n$ be the projection. We claim that
\begin{equation}\label{intersection}
\bigcap_{j=1}^k\pi\left(\Phi_1^{-1}(Y_j\setminus Y_\emptyset)\right)=\emptyset
\end{equation}
if $n<\sum_jd_j$. Let $y\in\pi\left(\Phi_1^{-1}(Y_j\setminus Y_\emptyset)\right)$ for all $j$, so that $y$ is an element of this intersection. For each $j$, we may choose $t_j$ so that $(y,t_j)\in\Phi^{1-}(Y_j\setminus Y_{\emptyset})$, so that $y=\pi(y,t_j)$ and $w(j)=(y,t_j)\in W_j$ for some cube $W_j\subset I^n\times I$. Thus, for each $j$, $w(j)\in W_j\cap \Phi_1^{-1}(Y_j\setminus Y_\emptyset)\subset K^{j,k}_{d_j+1}$ by (3) of \refT{BMpretheorem}. This means $w(j)$ has at least $d_j+1$ coordinates $w(j)_i$ such that $a_i+\frac{\delta(j-1)}{k}<w(j)_i<a_i+\frac{\delta j}{k}$, where $W_j=W(a,\delta,L)$ and $a=(a_1,\ldots, a_{n+1})$. This implies that $y$ has at least $d_j$ coordinates $y_i$ satisfying the same bounds (only here the index $i$ ranges between $1$ and $n$). For each $j$, the projection $\pi(W_j)$ is a cube containing $y$, and subdividing further if necessary (\refT{BMpretheorem} clearly still applies to any such further subdivision) we may assume $\pi(W_j)=W$ for all $j$. Thus $y$ has at least $d_j$ coordinates $y_i$ satisfying $a_i+\frac{\delta(j-1)}{k}<y_i<a_i+\frac{\delta j}{k}$ for all $j$, which is impossible if $n<\sum_j d_j$, so that the intersection in Equation (\ref{intersection}) is indeed empty. Hence there is some $i(y)\in\underline{k}$ so that $y\notin\pi\left(\Phi_1^{-1}(Y_{i(y)}\setminus Y_\emptyset)\right)$. That is, for all $t$, $(y,t)\notin \Phi_1^{-1}(Y_{i(y)}\setminus Y_\emptyset)$, as required. 

When $n=0$, to show $\pi_0(B)\to \pi_0(A)$ is surjective our argument above requires $d_j\geq 1$ for at least one $j$. We have already noted near the beginning of the proof that the conclusion of the theorem was trivially true when $d_j=0$ for all $j$.

%
%Write $y=\pi(z)$ for any suitable $z\in I^n\times I$ (of course, $z=(y,t)$ for some $t$). Since $z\in W$ for some cube $W$, $z\in W\cap \Phi_1^{-1}(Y_j\setminus Y_\emptyset)$ for every $j$, so by the third conclusion of \refT{BMpretheorem} (item (3) above), $z\in K^{j,k}_{d_j+1}(W)$ for each $j$. That is, for each $j=1$ to $k$, $z$ has at least $d_j+1$ coordinates $z_l$ so that $\frac{j-1}{k}<z_l<\frac{j}{k}$, which implies that $y$ has at least $d_j$ coordinates $y_l$ with the same bounds. This is impossible if $n<\sum_j d_j$, and so the intersection in Equation (\ref{intersection}) is indeed empty. Hence there is some $i(z)\in\underline{k}$ so that $y\notin\pi\left(\Phi_1^{-1}(Y_{i(z)}\setminus Y_\emptyset)\right)$; that is, for all $t$, $z=(y,t)\notin \Phi_1^{-1}(Y_{i(z)}\setminus Y_\emptyset)$. 
\end{proof}

%%%%%%%%%%%%%%%%%%%%%%%%%%%%%%%%%%%%%%%%%%%%%%%%%%%%%%%%%%%%%%%%%%%%%%%%%%%%%%%%%%%%%%%%%%%%%%%%%%%%%%%%

\section{Acknowledgments}\label{S:Ack}

%%%%%%%%%%%%%%%%%%%%%%%%%%%%%%%%%%%%%%%%%%%%%%%%%%%%%%%%%%%%%%%%%%%%%%%%%%%%%%%%%%%%%%%%%%%%%%%%%%%%%%%%

The author would like to thank the referee for a careful reading and thoughtful comments which  improved the exposition of this paper, and for pointing out a gap in the proof of \refL{Goodwillielemma}. This work was partially supported by ONR grants N0001412WX30191 and N0001413WX20992.

\bibliographystyle{amsplain}

\bibliography{/Users/brianmunson/Documents/TeX/Papers/Bibliography}

\end{document}